\documentclass[12pt]{amsart}

\usepackage{amsmath,amsfonts,amssymb,amsthm}
\usepackage{graphics,graphicx,color}
\include{showlabels}
\newtheorem{theorem}{Theorem}[section]
\newtheorem{lemma}[theorem]{Lemma}
\newtheorem{proposition}[theorem]{Proposition}

\newtheorem{definition}[theorem]{Definition}

\def \Rn {{\mathbb R}^n}
\def \Rone {{\mathbb R}}

\def \A {\mathcal{A}}
\def \tA {\tilde{\mathcal{A}}}

\def \D {{\mathcal D}}
\def \L {{\mathcal L}}
\def \J {{\mathcal J}}
\def \P {{\mathcal P}}

\def \p {\partial}
\def \st {\::\:}

\DeclareMathOperator{\diam}{diam}

\newcommand{\geov}[2]{\vec{\gamma}_{(#1)}(#2)}
\newcommand{\geo}[2]{\gamma_{(#1)}(#2)}
\newcommand{\dgeo}[2]{\dot\gamma_{(#1)}(#2)}

\newcommand{\ip}[1]{\langle #1\rangle}

\newcommand{\pp}[1]{\frac{\p}{\p#1}}

\newcommand{\eps}{\varepsilon}

\newcommand{\norm}{\|}

\newcommand{\tF}{\tilde{F}}
\newcommand{\ta}{\tilde{a}}
\newcommand{\tk}{\tilde{k}}

\begin{document}
\title{Stability of the Gauge Equivalent Classes in Inverse Stationary
Transport in Refractive Media}

\author{Stephen McDowall}
\thanks{First author partly supported by  NSF Grant No.~0553223}
\address{\hskip-\parindent
Stephen McDowall\\Department of Mathematics\\
Western Washington University\\
516 High Street\\ Bellingham, WA 98225-9063}
\email{stephen.mcdowall@wwu.edu}
\author{Plamen Stefanov}
\thanks{Second author partly supported by  NSF Grant No.~0554065}
\address{\hskip-\parindent
Plamen Stefanov\\Department of Mathematics\\
Purdue University\\
150 N. University Street\\
West Lafayette, IN 47907-2067} \email{stefanov@math.purdue.edu}
\author{Alexandru Tamasan}
\thanks{Third author partly supported by  NSF Grant No.~0905799}
\address{\hskip-\parindent
Alexandru Tamasan\\
Department of Mathematics\\
University of Central Florida\\
4000 Central Florida Blvd.\\Orlando, FL, 32816, USA}
\email{tamasan@math.ucf.edu}

\begin{abstract}In the inverse stationary transport problem through
anisotropic attenuating, scattering, and refractive media, the
albedo operator stably determines the gauge equivalent class of the
attenuation and scattering coefficients.
\end{abstract}

\maketitle

\pagestyle{myheadings}
\markboth{S. McDowall, P. Stefanov and A. Tamasan}{Stability of the
gauge equivalent classes}

\section{Introduction}
\label{sec:intro}

This paper concerns the problem of recovering the absorption and
scattering properties of a refractive medium from boundary knowledge
of the albedo operator. The medium $M\subset\Rn$, $n\geq 2$, is a
bounded domain with smooth boundary, endowed with a known Riemannian
metric $g$. The free moving particles travel through $M$ along the
geodesics. In the stationary case the propagation of particles is
modeled by the linear transport equation
\begin{align}
  \label{TE}
  -\D u(x,v)-a(x,v)u(x,v)+\int_{S_xM}k(x,v',v)u(x,v')\:d\omega_x(v')=0.
\end{align}

In the equation above $u(x,v)$ denotes the density of particles at
position $x$ with velocity $v$ in $S_xM$, the unit tangent sphere at
$x$. The operator $\D$ is the derivative along the geodesic flow:
For a given point $(x,v)\in S_x M$, if $\geo{x,v}{\cdot}$ denotes
the geodesic starting at $\geo{x,v}{0}=x$ with initial velocity
$\dgeo{x,v}{0}=v$, then
\begin{align}
  \label{geodesic flow}
  \D u(x,v):=\pp{t}\Bigr|_{t=0}u(\geov{x,v}{t}),
\end{align}
where, for brevity, we use the notation
$\geov{x,v}{t}=(\geo{x,v}{t},\dgeo{x,v}{t}).$ If $g$ is Euclidean then $\D$
is the directional derivative: $\D u(x,v)=v\cdot\nabla_xu(x,v)$. The
measure $d\omega_x(v')$ in \eqref{TE} is the volume form on $S_xM$ induced
from the volume form on $T_xM$ (the tangent space to $M$ at $x$) determined
by $g$ at $x$. The resulting (Liouville) form on $S M$ is preserved under
the geodesic flow of $g$, see \cite{Sh94}. The attenuation coefficient
$a(x,v)$ in \eqref{TE} quantifies the rate at which particles are lost from
the point $(x,v)$ in phase space due to absorption and scattering into new
directions.  The scattering coefficient $k(x,v',v)$ represents the
probability that a particle at position $x$ with velocity $v'\in S_xM$ will
scatter to have new velocity $v\in S_xM$.

The boundary measurements are described by the albedo operator $\A$: Let
$\Gamma_{\pm}=\p_\pm SM=\{(x,v)\in \p SM\st\pm\ip{v,\nu_x}>0\}$ denote the
``incoming'' and ``outgoing'' bundles, where $\nu_x$ is the unit outer
normal vector to the boundary $\p M$ at $x$ and $\ip{\cdot,\cdot}$ is the
inner product, each with respect to $g$ at $x$.  The medium is probed with
the given radiation
\begin{align}
  \label{BC}
  u|_{\Gamma_-}=u_-,
\end{align}
and the exiting radiation is detected on $\Gamma_+$. The albedo
operator takes the incoming flux to the outgoing flux at the
boundary: $\A u_-=u|_{\Gamma_+}$.

The inverse boundary value problem is to determine the coefficients
$a$ and $k$ from the knowledge of $\A$.

When the attenuation $a$ is isotropic ($v$-independent), there is a large
collection of uniqueness results under varying assumptions on the
parameters; see \cite{Ba09} for a comprehensive account, and \cite{Kui10}
for a recent survey of numerical methods. The works below are based on the
singular decomposition of the Schwartz kernel of $\A$, an idea first
introduced in \cite{ChSt96} and \cite{ChSt99}; see also \cite{Bo}. In the
Euclidean setting, uniqueness of $a(x)$ (dimensions two and above) and
$k(x,v',v)$ (dimensions three and above) was proven in \cite{ChSt99} under
some minimal restrictions guaranteeing only that the forward problem is
well-posed. For sufficiently small $k$ this result was extended to
dimension two in \cite{StUh03}; see also \cite{Ta02,Ta03}. Stability
results are proven in \cite{StUh03,La08,Ro66,Wa99} with the most general
result (in Euclidean geometry) in \cite{BaJo08}. When no angular resolution
is measured in the outgoing flux, the singular decomposition of the new
boundary operator has been used to recover an isotropic coefficient $a$ and
the spatial part of $k$ in \cite{La08, BaLaMo08}.

The case of a Euclidean metric corresponds to transport in
materials with a constant index of refraction.  If the index of
refraction is isotropic, but varying, then \eqref{TE} can be derived
as a limiting case of Maxwell's equations with non-constant (but
isotropic) permeability, resulting in a metric which is conformal to
the Euclidean metric (\cite{Ba06}). For a general metric, we
consider \eqref{TE} as a model for transport in a medium with
varying, anisotropic index of refraction. When the attenuation is
assumed isotropic, uniqueness results in Euclidean geometry are
extended to the Riemannian metric setting in
\cite{LaMc08,Mc04,Mc05,Mc08}. There and here, the manifold is
assumed to be {\em simple} as follows.

{\bf Definition:} {\it $(M,g)$ is called {\em simple} if it is
strictly convex, and for any $x\in \overline M$ the exponential map
$\exp_x:\exp_x^{-1}(\overline M)\to\overline M$ is a diffeomorphism.
If $M$ is two dimensional we have the following additional
assumption: Let $\kappa$ be the maximum sectional curvature of $M$.
If $\kappa>0$, then we also assume $\diam(M)<\pi/\sqrt{\kappa}$.}

The works mentioned above concern the media with an isotropic
attenuation character. However, since the attenuation is a
combination of absorption and loss of particles due to scattering:
$a(x,v)=\sigma(x,v)+\int_{S_xM}k(x,v,v')\:d\omega_x(v')$, even when
the absorption part is isotropic ($\sigma=\sigma(x)$), if $k$
depends on two independent directions the resulting attenuation is
anisotropic. Evidence of anisotropy in biological tissue has been
observed experimentally, see \cite{Wang98}.

When the attenuation coefficient is anisotropic, it is possible to
have media of differing attenuation and scattering properties which
yield the same albedo operator. Moreover the non-uniqueness is
characterized by the action of a gauge transformation \cite{StTa09}:
see \eqref{gauge} below. The same algebraic structure of
non-uniqueness is valid in refractive media \cite{McStTa10-GaEq}.

In Theorems \ref{main_thm} and \ref{main_thm_2d} we show the
stability of the gauge equivalent classes occurring in refractive
media, thus extending the results from the Euclidean case in
\cite{McStTa10-GaSt}. We also generalize a stability result in
\cite{BaJo08} from Euclidean to Riemannian geometry.

The algebraic structure of non-uniqueness in
\cite{StTa09,McStTa10-GaEq} can be readily observed. Indeed, if
$\phi\in L^\infty(S M)$ is positive with $1/\phi\in L^\infty(S M)$,
$\D\phi\in L^\infty(S M)$ and such that $\phi=1$ on $\p SM$. Set
\begin{align}
  \label{gauge}
  \tilde a(x,v)=a(x,v)-\D\log\phi(x,v),\quad
  \tilde k(x,v',v)=\frac{k(x,v',v)\phi(x,v)}{\phi(x,v')}.
\end{align}
Then $u$ satisfies \eqref{TE} if and only if $\tilde{u}=\phi u$
solves
\begin{align*}
  -\D\tilde u(x,v)-\tilde a(x,v)\tilde u(x,v)
  +\int_{S_xM}\tilde k(x,v',v)\tilde u(x,v')\:d\omega_x(v')=0.
\end{align*}
Since $\phi=1$ on $\Gamma$, $u=\tilde u$ there, so the albedo
operator $\A$ for the parameters $(a,k)$ is indistinguishable from
the albedo operator $\tilde \A$ for the pair $(\tilde a,\tilde k)$,
i.e. $\A=\tilde \A$. This motivates the following definition in
\cite{StTa09,McStTa10-GaEq}.

\begin{definition}\label{gauge_equivalence}
Two pairs of coefficients $(a,k)$ and $(\tilde{a},\tilde{k})$ are
called
  {\em gauge equivalent} if there exists a positive map $\phi\in L^\infty(S
  M)$ with $1/\phi\in L^\infty(S M)$, $\D\phi\in L^\infty(S M)$, and
  $\phi=1$ on $\Gamma$, such that \eqref{gauge} holds. We denote this
  equivalence by $(a,k)\sim (\tilde{a},\tilde{k})$.
\end{definition}

The relation defined above is reflexive since $(a,k)\sim (a,k)$ via
$\phi\equiv 1$; it is symmetric since $(a,k)\sim
(\tilde{a},\tilde{k})$ via $\phi$ yields $(\tilde{a},\tilde{k})\sim
(a,k)$ via $1/\phi$; and it is transitive since if $(a,k)\sim
(\tilde{a},\tilde{k})$ via $\phi$ and $(\tilde{a},\tilde{k})\sim
(a',k')$ via $\tilde{\phi}$ then $(a,k)\sim (a',k')$ via
$\phi\tilde{\phi}$. Therefore one has the multiplicative group of
gauges acting transitively (since any equivalent pair are related by
some gauge $\phi$ \cite{StTa09,McStTa10-GaEq}) on the equivalent
class of a pair of coefficients. We denote the equivalence class of
$(a,k)$ by $\langle a,k\rangle$.

\section{Transport of the data to a larger domain}\label{extension}
Due to the method of proof, the total travel time of each particle
in $M$ has to be uniformly bounded away from zero. This can be done
without loss of generality by doing the measurements away from the
boundary $\partial M$. More precisely, let $M_0$ be a slightly
larger domain strictly containing $M$. The metric $g$ can be
extended to $g_0$ on $M_0$ in such a way that $(M_0,g_0)$ still
remains simple \cite{StUh04}.

As in \cite{McStTa10-GaSt}, we reduce the problem in $M$ to one in
$M_0$: Let $(a,k)$ and $(\ta,\tk)$ be coefficients for which the
forward problems in $(M,g)$ are well-posed, and $\A$ and $\tA$ be
their corresponding albedo operators.  Defining $a=\ta=k=\tk=0$ in
$M_0\setminus M$, the forward problems in $(M_0,g_0)$ are also
well-posed and the albedo operators $\A_0$ and $\tA_0$ are well
defined maps between functions on
\begin{align*}
  \Gamma_\pm^0:=\p_\pm SM_0=\{(x,v)\in \p
  SM_0\st\pm\ip{v,\nu_x}>0\}
\end{align*}
(now $\nu_x$ is the outer unit normal vector to $\p M_0$ at $x$,
with respect to the extended metric $g_0$).  As in
\cite{McStTa10-GaSt}, when the two forward problems for $M$ are
well-posed in $L^p$, $1\leq p\leq \infty$, the following isometric
property holds:
\begin{align}
  \label{isometry}
  \|\A-\tA\|_{\L(L^p(\Gamma_-,d\mu);L^p(\Gamma_+,d\mu))}
  =\|\A_0-\tA_0\|_{\L(L^p(\Gamma_-^0,d\mu_0);L^p(\Gamma_+^0,d\mu_0))}.
\end{align}The measure $d\mu$ (and, analogously, $d\mu_0$) in \eqref{isometry}
is defined as follows: Let $d\Sigma^{2n-2}$ be the volume form on
$\Gamma_{\pm}$ obtained by the natural restriction of the volume
form on $SM$ to $\Gamma_\pm$. Then, by extending $d\Sigma^{2n-2}$ as
a homogeneous form of order $n-1$  in $|v'|$, we have that
$d|v'|\:d\Sigma^{2n-2}(x',v')$ coincides with the volume form on
$SM$; see \cite{Sh94} for details. We define
\begin{align}\label{dmu}
  d\mu(x',v')=|\ip{v',\nu_{x'}}|\:d\Sigma^{n-2}(x',v').
\end{align}
{\bf Remark:} The proof of \eqref{isometry} is essentially identical to
that in the Euclidean case as presented in \cite{McStTa10-GaSt} due to the
following invariant property of the form $d\mu$: Fix $(x'_0,v'_0)\in
\Gamma_\pm$. Let $\partial\tilde\Omega$ be any surface so that the geodesic
issued from $(x_0',v_0')$ hits it transversally. Then the geodesic flow
defines a natural local ``projection'' near $(x'_0,v'_0)$, of $\Gamma_\pm$
onto $\tilde\Gamma_\pm$. Let $d\tilde \mu$ be the measure on
$\tilde\Gamma_\pm$ defined analogously to \eqref{dmu}. Then the pull back
of $d\tilde \mu$ is $d\mu$ (see the proof of Lemma 4.2.2 in \cite{Sh94}).
Given $(x,v)\in SM$, we define the ``forward/backward travel time to the
boundary'' functions
\begin{align*}
  \tau_\pm(x,v)=\min\{t\geq0\st\geo{x,v}{\pm t}\in\p M_0\},
\end{align*}
and let $\tau(x,v)=\tau_+(x,v) +\tau_-(x,v)$ be the total travel
time of the free particle (of velocity $v$ at $x$) through $M_0$.
Since $\text{dist}_{g_0}(\overline{M},\p M_0)>0$, we have
\begin{align}
  \label{nontangential}
  c_0:=\inf\{\tau(x,v)\st (x,v)\in\overline{SM}\}>0.
\end{align}

Using the isometry property \eqref{isometry}, we can considered the
inverse problem in the larger domain $(M_0,g_0)$ with the albedo
operators now acting between $\Gamma_\pm^0$. Equivalently, for the
original problem in $(M,g)$ we may work without loss of generality
with coefficients $a,k$ of (a priori fixed) compact support in $M$.

To simplify notation, while still working in the larger domain, we
drop the $0$ index throughout the remaining of the paper: thus $M_0$
becomes $M$, $\Gamma_\pm^0$ becomes $\Gamma_\pm$, etc.

\section{The singular structure of the albedo operator's kernel}
\label{preliminaries} In this section we recall the singular
decomposition of the Schwartz kernel of the albedo operator for the
two cases separated by dimension.

We work within the class of {\em admissible} coefficients: For
$n\geq 3$
\begin{align}\label{continuous_coeff}
  (a,k)\in L^\infty(SM)\times L^\infty(SM,L^1(S_xM)),
\end{align}and for $n=2$
\begin{align}\label{bounded_coeff_2d}
  (a,k)\in L^\infty(SM)\times L^\infty(S^2M),
\end{align}where
$S^2M:=\{(x,v',v)\st x\in M,\ v',v\in S_xM\}.$ Note that the gauge
transformations \eqref{gauge} preserve the admissible classes in
\eqref{continuous_coeff} and \eqref{bounded_coeff_2d}.

Moreover, either one of the following {\em subcritical} conditions
that yield well-posedness for the boundary value problem \eqref{TE}
and \eqref{BC} is assumed to hold:
\begin{equation}
  \label{critic_CS}
  \text{ess sup}_{(x,v)\in SM}
  \left|\tau(x,v)\int_{S_xM}k(x,v,v')\:d\omega_x(v')\right|<1,
\end{equation}
or
\begin{equation}
  \label{critic_DL}
  a(x,v)-\int_{S_xM}k(x,v,v')\:d\omega_x(v')\geq0,
  \quad a.e.\ (x,v)\in SM;
\end{equation}
see, e.g., \cite{BaJo08,ChSt99, DaLi93, MoKh97, ReSi79}.

The right hand side of \eqref{TE} defines a closed, unbounded operator on
$L^1(SM_R)$ with the domain $\{u\in L^1(SM_R)\st \D u\in L^1(SM_R),\;
u|_{\Gamma_-}=0\}$; see \cite{ChSt99,McStTa10-GaEq}.

Proposition \ref{albedo kernel 3} below, which describes the terms
in the expansion of the kernel of $\A$, is proven in \cite{Mc04},
the Euclidean equivalent appearing in \cite{ChSt99}.  We denote by
$\delta_{\{x',v'\}}(x,v)$ the delta-distribution on $\Gamma_+$ with
respect to the measure $d\mu$ defined by
\begin{align*}
  \int_{\Gamma_+}\varphi(x,v)\delta_{\{x',v'\}}(x,v)\:d\mu(x,v)
  =\varphi(x',v'),\qquad
  \varphi\in C_c^\infty(\Gamma_+).
\end{align*}
Similarly, $\delta_{\{x\}}(y)$ is the delta distribution on $M$
supported at $x$.

If $x,y\in M$ let $v(x,y)\in S_xM$ denote the
tangent vector at $x$ of the unit speed geodesic joining $x$ to $y$
(uniquely defined since $M$ is simple), and $d(x,y)$ be the
Riemannian distance between $x$ and $y$. Denote the total
attenuation along the geodesic from $x$ to $y$ by
\begin{align}\label{attenxtoy}
  E(x,y):=\exp\Bigl\{-\int_0^{d(x,y)}a\bigl(\geov{x,v(x,y)}{t}\bigr)
  \:dt\Bigr\}.
\end{align}
Note that $\dgeo{y,v(y,x)}{d(y,x)-s}=-\dgeo{x,v(x,y)}{s}$, so when
$a$ depends on direction, $E(x,y)\neq E(y,x)$.

\begin{proposition}
   \label{albedo kernel 3}
   \cite{Mc04} Let $(M,g)$ be a smooth simple Riemannian manifold of
   dimension $n\geq3$.  Assume that $(a,k)$ are admissible and subcritical
   so that the forward problem is well-posed.  Then the albedo operator
   $\A:L^1(\Gamma_-,d\mu)\to L^1(\Gamma_+,d\mu)$ is bounded and its
   Schwartz kernel $\alpha(x,v,x',v')$, considered as a distribution on
   $\Gamma_+$ parameterized by $(x',v')\in\Gamma_-$, has the expansion
   $\alpha=\alpha_0+\alpha_1+\alpha_2$, where
  \begin{align}
    \alpha_0&=E\bigl(\geo{x,v}{-\tau_-(x,v)},x\bigr)
    \delta_{\{\geov{x',v'}{\tau(x',v')}\}}(x,v),\label{alpha0}\\
    \alpha_1&=\int_0^{\tau_+(x',v')}\int_0^{\tau_-(x,v)}
    E(y(s),x)E(x',z(t))
    k(z(t),\dot z(t),\dot y(s))
    \delta_{\{y(s)\}}(z(t))\:ds\:dt\label{alpha1}\\
    &\qquad y(s)=\geo{x,v}{s-\tau_-(x,v)},\ z(t)=\geo{x',v'}{t},\nonumber\\
    \alpha_2&\in L^\infty(\Gamma_-;L^1(\Gamma_+,d\mu)).
    \label{alpha2}
  \end{align}
\end{proposition}
We note that $k(z(t),\dot z(t),\dot y(s))$ is only defined on the
support of the integrand, namely when $y(s)=z(r)$.

When $n=2$ the left hand side of \eqref{TE} defines an unbounded operator
of domain $ \{u\in L^\infty(S M)\st \D u\in L^\infty(S M),\ 
u|_{\Gamma_-}=0\}.$ Provided that \eqref{bounded_coeff_2d} holds, and we
have subcriticality \eqref{critic_CS} or \eqref{critic_DL}, this operator
has a bounded inverse in $L^\infty(S M)$, see \cite{Mc05}, and the singular
decomposition of the albedo kernel is more explicit as follows.

Given $(x,v,x',v')\in \Gamma_+\times\Gamma_-$, define
$\chi:\Gamma_+\times\Gamma_-\to\{0,1\}$ by $\chi(x,v,x',v')=1$ if
there exist $0\leq s=s(x,v,x',v')\leq \tau_-(x,v)$ and $0\leq
t=t(x,v,x',v')\leq \tau_+(x',v')$ such that
$\geo{x,v}{s-\tau_-(x,v)}=\geo{x',v'}{t}$ (i.e., the geodesics
intersect in $M$), and $\chi(x,v,x',v')=0$ otherwise. When
$\chi(x,v,x',v')=1$, let $\psi(x,v,x',v')$ be the angle between the
tangent vectors of these geodesics at the point of intersection.
\begin{proposition}
   \label{albedo kernel 2}
   \cite{Mc05} Let $(M,g)$ be a two dimensional simple Riemannian
   manifold. Assume that $(a,k)$ are
   admissible and that \eqref{bounded_coeff_2d} holds.  Then the albedo
   operator $\A:L^\infty(\Gamma_-,d\mu)\to L^\infty(\Gamma_+,d\mu)$ is
   bounded and its Schwartz kernel $\alpha(x,v,x',v')$, considered as a
   distribution on $\Gamma_+$ parameterized by $(x',v')\in\Gamma_-$, has
   the expansion $\alpha=\alpha_0+\alpha_1+\alpha_2$, where
  \begin{align*}
    &\alpha_0=E\bigl(\geo{x,v}{-\tau_-(x,v)},x\bigr)
    \delta_{\{\geov{x',v'}{\tau(x',v')}\}}(x,v),\\
    &\alpha_1=\chi(x,v,x',v')E(x',\geo{x',v'}{t})E(\geo{x',v'}{t},x)\J
    \frac{k\bigl(\geov{x',v'}{t},\dgeo{x,v}{s-\tau_-(x,v)}
      \bigr)}{|\sin(\psi(x,v,x',v'))|},\\
    &0\leq \alpha_2\chi\leq C\|k\|^2_{L^\infty(S^2M)}
    \Bigl(1+\log\frac{1}{|\sin(\psi(x,v,x',v'))|}\Bigr).
  \end{align*}
  Here, $\J=\J(x,v,x',v')$ is a function uniformly bounded $0<m_1\leq
  \J\leq m_2<\infty$ on $\Gamma_+\times\Gamma_-$ (see \cite[Proposition
  4]{Mc05}).
\end{proposition}

\section{Statement of the main results}
Let $(B_a,\|\cdot\|_{B_a})$ and $(B_k,\|\cdot\|_{B_k})$ be Banach
spaces in which the attenuation and, respectively, the scattering
kernel are considered, $(a,k),(\ta,\tk)\in B_a\times B_k$. The {\em
distance} $\Delta$ between equivalence classes {\em with respect to
$B_a\times B_k$ }is defined by the infimum of the distances between
all possible pairs of representatives:
\begin{align*}
\Delta(\langle a,k\rangle,\langle
\ta,\tk\rangle):=\inf_{(a',k')\in\langle a,k\rangle,~~
(\ta',\tk')\in\langle\ta,\tk\rangle}\max\{\|a'-\ta'\|_{B_a},
\|k'-\tk'\|_{B_k}\}.
\end{align*}

The following norms are used throughout
\begin{align*}
  \|a\|_\infty&= \text{ess sup}_{(x,v)\in SM}|a(x,v)|,\\
  \|k\|_{\infty,1}
  &=\text{ess sup}_{(x,v')\in SM}\int_{S_xM}|k(x,v',v)|\:d\omega_x(v),\\
  \|k\|_\infty
  &=\text{ess sup}_{(x,v',v)\in S^2M}\left|k(x,v',v)\right|,\\
  \|k\|_1&=\int_M\int_{S_xM}\int_{S_xM}|k(x,v',v)|
  \:dx\:d\omega_x(v')\:d\omega_x(v).
\end{align*}

{\bf Case $n\geq3$:} Define the class
\begin{align*}
  U_{\Sigma,\rho}:=\{(a,k)\text{ as in }
  \eqref{continuous_coeff}\st\|a\|_\infty\leq\Sigma,\
  \|k\|_{\infty,1}\leq\rho\}.
\end{align*}

\begin{theorem}
  \label{main_thm}
  Let $(M,g)$ be simple Riemannian manifold of dimension $n\geq 3$. Let
  $(a,k),(\ta,\tk)\in U_{\Sigma,\rho}$ be such that the corresponding
  forward problems are well posed. Then
  \begin{align*}
    \Delta(\langle a,k\rangle,\langle \ta,\tk\rangle)\leq
    C\|\A-\tA\|_{\L(L^1(\Gamma_-,d\mu);L^1(\Gamma_+,d\mu))},
  \end{align*}
  where $\Delta$ is with respect to $L^\infty(SM)\times L^1(S^2M)$, and $C$
  is a constant depending only on $\Sigma$, $\rho$, $n$, and $c_0$ in
  \eqref{nontangential}. More precisely, there exists a representative
  $(a',k')\in\langle a,k\rangle$ such that
  \begin{align}
    \|a'-\ta\|_\infty
    &\leq C\|\A-\tA\|_{\L(L^1(\Gamma_-,d\mu);L^1(\Gamma_+,d\mu))},
    \label{closeness_in_a}\\
    \|k'-\tk\|_1
    &\leq C\|\A-\tA\|_{\L(L^1(\Gamma_-,d\mu);L^1(\Gamma_+,d\mu))}.
    \label{closeness_in_k}
  \end{align}
\end{theorem}

{\bf Case $n=2$:} From Proposition \ref{albedo kernel 2} above
recall the Schwartz kernel of the albedo operator in the form
\[\alpha=
A_0(x,v)\delta_{\{\geov{x',v'}{\tau(x',v')}\}}(x,v)+\beta\] where
\begin{align*}
&A_0(x,v)=E\bigl(\geo{x,v}{-\tau_-(x,v)},x\bigr)\in L^\infty(\Gamma_+)\\
&\beta(x,v,x',v')\chi|\sin\psi(x,v,x',v')|
 \in L^\infty(\Gamma_+^R\times\Gamma_-^R).
\end{align*}
We define
\begin{align}\label{start_norm}
  \|\A\|_*=\max\{\|A_0\|_\infty,\|\beta\chi|\sin\psi|\|_\infty\}.
\end{align}
By using the Remark in Section \ref{extension}, the proof in
\cite{McStTa10-GaSt} carries through verbatim to show that
$\|\A-\tA\|_*$ is preserved when transported from the boundary of
the inner domain $M$ to the boundary of the larger domain.

Define the class
\begin{align*}
  V_{\Sigma,\rho}:=\bigl\{(a,k)\text{ as in \eqref{bounded_coeff_2d}}\st
  \|a\|_\infty\leq \Sigma,\ \|k\|_{\infty}\leq \rho\bigr\}.
\end{align*}

\begin{theorem}
  \label{main_thm_2d}
Let $(M,g)$ be a two dimensional simple Riemannian manifold.
For any $\Sigma>0$ there exists $\rho>0$
  depending only on $\Sigma$ and $(M,g)$ such that the following holds: if
  $(a,k)$, $(\tilde a,\tilde k)\in V_{\Sigma,\rho}$ then
  \begin{align*}
    \Delta(\ip{a,k},\ip{\ta,\tk})\leq C\|\A-\tA\|_*
  \end{align*}
  where $\Delta$ is with respect to $L^\infty(SM)\times L^\infty(S^2M)$ and $C$
  is a constant depending only on $\Sigma$ and $(M,g)$.
\end{theorem}
Note that $\rho$ sufficiently small already yields a subcritical
regime as in \eqref{critic_CS}.
\section{Preliminary estimates}\label{Preliminary Estimates}
In this section we extend a result from the Euclidean to Riemannian
setting; see \cite[Theorem 3.2]{BaJo08} for contrast. In addition,
the proof below allows for discontinuous coefficients, which is
needed when transporting the albedo operator to the larger domain
(and no boundary knowledge of the coefficients is available).

\begin{lemma} \label{app_id_cor}
  There is a family of maps $\phi_{\eps,x_0',v_0'}\in
  L^1(\Gamma_-,d\mu)$, for $(x_0',v_0')\in\Gamma_-$ and $\eps>0$, such that
  $\|\phi_{\eps,x_0',v_0'}\|_{L^1(\Gamma_-,d\mu)}=1$ and,
  for any $f\in L^\infty(\Gamma_-,d\mu)$ given,
  \begin{align}\label{app_id_limit}
    \lim_{\eps\to 0}\int_{\Gamma_-}
    \phi_{\eps,x_0',v_0'}(x',v')f(x',v')\:d\mu(x',v')
    =f(x_0',v_0'),
  \end{align}
  whenever $(x_0',v_0')$ is in the Lebesgue set of $f$. In particular,
  \eqref{app_id_limit} holds for almost every $(x_0',v_0')\in\Gamma_-$.
\end{lemma}
For a measurable function $f$ on $\Rn$, The Lebesgue set of $f$ is
\begin{align*}
  L_f:=\bigl\{x\st\lim_{r\to0}\frac{1}{|B_r(x)|}
  \int_{B_r(x)}|f(y)-f(x)|\:dy=0\bigr\}
\end{align*}
where $B_r(x)$ is the ball of radius $r$ centered at $x$, and where
$|\cdot|$ denotes Lebesgue measure.  We point out that if $\in
L^1_{loc}(\Rn)$, then $|\Rn\setminus L_f|=0$ (\cite[Theorem 3.20]{folland}).
\begin{proof}
  For $(x_0',v_0')\in \Gamma_-$ and $\eps>0$ sufficiently small, let
  $(x',v'):U\times W\subset\Rone^{n-1}\times\Rone^{n-1}\to\p SM$ with
  $x'(U)\subset\p M$ be a coordinate chart near
  $(x_0',v_0')=(x'(0),v'(0))$.  Let
  $d\Sigma^{2n-2}(x',v')=\sqrt{g_-}\:du\:dw$ be the local coordinate
  expression for the volume element (see \eqref{dmu}).  For
  $(x',v')\in\Gamma_-$, define
  \begin{align*}
    \phi_{\eps,x_0',v_0'}(x',v')=\frac{1}{|\ip{\nu_{x'},v'}|}
    \frac{1}{\sqrt{g_-(x',v')}}
    \varphi_\eps(u(x'))\varphi_\eps(w(x',v')),
  \end{align*}
  where $\varphi(u)\equiv 1/(\omega_{n-1})$ for $|u|< 1$, $\varphi(u)\equiv
  0$ for $|u|\geq 1$, and $\varphi_\eps(u)=\eps^{-n+1}\varphi(u/\eps)$. By
  $\omega_{n-1}$ we denoted the volume of the unit ball in $R^{n-1}$.
  Then, for any $\eps>0$, $\int\varphi_\eps(u)du=1$ and, by using
  \eqref{dmu}, we obtain
  \begin{align*}
    \int_{\Gamma_-}\phi_{\eps,x_0',v_0'}(x',v')&f(x',v')\:d\mu(x',v')\\
    &=\int_{\Gamma_-}\varphi_\eps(u(x'))\varphi_\eps(w(x',v'))f(x',v')
    \frac{1}{\sqrt{g_-(x',v')}}\:d\Sigma^{n-2}(x',v')\\
    &=\int_{\Rone^{2n-2}}\varphi_\eps(u)\varphi_\eps(w)f(x'(u),v'(u,w))
    \:du\:dw.
  \end{align*}
  Apply the equality above to $f\equiv 1$ to get
  $\|\phi_{\eps,x_0',v_0'}\|_{L^1(\Gamma_+,d\mu)}=1$. The conclusion
  follows from the approximation of identity for $L^p$ maps, see, e.g.,
  \cite[Theorem 8.15]{folland}.
\end{proof}

Consider $F:\p M\times SM\to\Rone$ defined by
\begin{align}
  \label{F defn}
  F(x',y,w)=E(x',y)E\bigl(y,\geo{y,w}{\tau_+(y,w)}\bigr),
\end{align}with $E$ as in \eqref{attenxtoy}.  $F(x',y,w)$ represents the
total attenuation along the broken geodesics from $x'\to y\to
\geo{y,w}{\tau_+(y,w)}$. 

Let $(a,k),(\tilde{a},\tilde{k})$ be admissible pairs as in
\eqref{continuous_coeff}. Recall that, after the extension of the
domain, the coefficients have fixed compact support away from the
boundary $\partial M$. All the operators bearing the tilde refer to
$(\tilde{a},\tilde{k})$ and are defined in a similar way to the ones
for $(a,k)$, i.e., $\tilde{\mathcal{A}}$ is the albedo operator
corresponding to $(\tilde{a},\tilde{k})$. Recall that $n$ is the
dimension of the space. To simplify notation let
$$\|\A-\tA\|:=\|\A-\tA\|_{\L(L^1(\Gamma_-,d\mu);L^1(\Gamma_+,d\mu))}.$$

\begin{theorem}
  \label{bal_lemma}
  Let $(a,k),(\tilde{a},\tilde{k})$ be as in \eqref{continuous_coeff}. For
  almost every $(x_0',v_0')\in\Gamma_-$ the following estimates hold: For
  $n\geq 2$,
  \begin{align}
    \label{bal_estimate1}
    \Bigl|e^{-\int_0^{\tau_+(x_0',v_0')}a(\geov{x_0',v_0'}{s})\:ds}-
    e^{-\int_0^{\tau_+(x_0',v_0')}\tilde{a}(\geov{x_0',v_0'}{s})\:ds}\Bigr|
    \leq\norm\A-\tA\norm.
  \end{align}
  For $n\geq 3$, with $y(t)=\geo{x_0',v_0'}{t}$,
  \begin{align}
    \int_0^{\tau_+(x_0',v_0')}\int_{S_{y(t)}M}
    &|k-\tilde{k}|(y(t),\dot y(t),w)
    F(x_0',y(t),w)\:d\omega_y(w)\:dt\nonumber\\
    &\leq\norm\A-\tA\norm+
    \|F-\tF\|_\infty\int_0^{\tau_+(x_0',v_0')}
    \int_{S_{y(t)}M}\tilde{k}(y(t),\dot y(t),w)\:d\omega_y(w)\:dt.
    \label{bal_estimate2}
  \end{align}
\end{theorem}
\begin{proof}
  Let $(x_0',v_0')\in\Gamma_-$ be arbitrarily fixed and let
  $\phi_{\eps,x_0',v_0'}\in L^1(\Gamma_-)$ be defined as in Corollary
  \ref{app_id_cor}. To simplify the formulas, since $(x_0',v_0')$ is fixed,
  in the following we drop this dependence from the notation
  $\phi_\eps=\phi_{\eps,x_0',v_0'}$.

  Let $\A=\A_0+\A_1+\A_2$ be the decomposition of the albedo operator given
  by
  \begin{align*}
    \A_if(x,v)=\int_{\Gamma_-}\alpha_i(x,v,x',v')f(x',v')\:d\mu(x',v'),
    \quad i=0,1,2,
  \end{align*}
  where $\alpha_i$, $i=0,1,2$ are the Schwartz kernels in Proposition
  \ref{albedo kernel 3}.

  Let $\phi\in L^\infty(\Gamma_+)$ with $\|\phi\|_\infty\leq 1$. Since
  $\|\phi_\eps\|_{L^1(\Gamma_-)}=1$, the mapping properties of the albedo
  operator imply that
  \begin{align}
    \label{original_estimate}
    \left|\int_{\Gamma_+}\phi(x,v)[\A-\tA]\phi_\eps(x,v)\:d\mu(x,v)\right|
    \leq\|\A-\tA\|.
  \end{align}
  Next we evaluate each of the three terms in
  $\int_{\Gamma_+}\phi(x,v)[\A-\tA]\phi_\eps(x,v)\:d\mu(x,v)$ by using the
  decomposition in Proposition \ref{albedo kernel 3} and Fubini's theorem.

  The first term is evaluated using the formula \eqref{alpha0}:
  \begin{align*}
    I_0(\phi,\eps)
    :=\int_{\Gamma_+}\phi(x,v)[\A_0-\tA_0]\phi_\eps(x,v)\:d\mu(x,v)
    =\int_{\Gamma_-}\phi\bigl(\geov{x',v'}{\tau_+(x',v')}\bigr)
    \phi_\eps(x',v')\\
    \times\Bigl[e^{-\int_0^{\tau_+(x',v')}a(\geov{x',v'}{s})\:ds}-
    e^{-\int_0^{\tau_+(x',v')}\tilde{a}(\geov{x',v'}{s})\:ds}\Bigr]
    \:d\mu(x',v').
  \end{align*}
  Since the integrand above is in $L^\infty(\Gamma_-)$ by applying
  \eqref{app_id_limit}, we get for almost every $(x_0',v_0')\in\Gamma_-$
  \begin{align}
    I_0(\phi)&(x_0',v_0')
    :=\lim_{\eps\to0}I_0(\phi,\eps)\nonumber\\
    &=\phi\bigl(\geov{x_0',v_0'}{\tau_+(x_0',v_0')}\bigr)
    \Bigl(e^{-\int_0^{\tau_+(x_0',v_0')}a(\geov{x_0',v_0'}{s})\:ds}-
    e^{-\int_0^{\tau_+(x_0',v_0')}\tilde{a}(\geov{x_0',v_0'}{s})\:ds}\Bigr).
    \label{term1_lim}
\end{align}

To evaluate the second term we use the formula \eqref{alpha1} and let
$y=y(x',v',t)=\geo{x',v'}{t}$:
\begin{align*}
  I_1(\phi,\eps)
  &:=\int_{\Gamma_+}\phi(x,v)[\A_1-\tA_1]\phi_\eps(x,v)\:d\mu(x,v)
  \nonumber\\
  &=\int_{\Gamma_-}\phi_\eps(x',v')\:d\mu(x',v')
  \Bigl\{\int_0^{\tau_+(x',v')}\int_{S_yM}
  \phi\bigl(\geov{y,w}{\tau_+(y,w)}\bigr)\nonumber\\
  &\qquad\qquad\times
  \left[F(x',y,w)k(y,\dot y,w)-\tF(x',y,w)\tk(y,\dot y,w)\right]
  \:dw\:dt\Bigr\}.
\end{align*}
Apply again \eqref{app_id_limit} for the continuous integrand above to
obtain for almost every $(x_0',v_0')\in\Gamma_-$: with $y=y(x_0',v_0',t)$
\begin{align}
  I_1(\phi)(x_0',v_0')
  &:=\lim_{\eps\to0}I_1(\phi,\eps)\nonumber\\
  &\phantom{:}=\int_0^{\tau_+(x_0',v_0')}\int_{S_yM}
  \phi\bigl(\geov{y,w}{\tau_+(y,w)}\bigr)\nonumber\\
  &\qquad\qquad\times
  \bigl[F(x_0',y,w)k(y,\dot y,w)-\tF(x_0',y,w)\tk(y,\dot y,w)\bigr]
  \:d\omega_y(w)\:dt,\label{term2_lim}
\end{align}
or $I_1(\phi)=I_{1,1}(\phi)+I_{1,2}(\phi)$ with
\begin{align}
  I_{1,1}(\phi)
  &=\int_0^{\tau_+(x_0',v_0')}\int_{S_yM}
  \phi\bigl(\geov{y,w}{\tau_+(y,w)}\bigr)
  F(x_0',y,w)(k-\tk)(y,\dot y,w)\:d\omega_y(w)\:dt,\label{term21_lim}\\
  |I_{1,2}(\phi)|
  &\leq\int_0^{\tau_+(x_0',v_0')}\int_{S_yM}
  |F-\tF|(x_0',y,w)\tk(y,\dot y,w)\:d\omega_y(w)\:dt\label{term22_lim}.
\end{align}
Consider the third term
\begin{align*}
  I_2(\phi,\eps)
  &=\int_{\Gamma_+}\phi(x,v)[\A_2-\tA_2]\phi_\eps(x,v)\:d\mu(x,v)\\
  &=\int_{\Gamma_-}\phi_\eps(x',v')\:d\mu(x',v')
  \Bigl\{\int_{\Gamma_+}\phi(x,v)(\alpha_2-\tilde{\alpha}_2)(x,v,x',v')
  \:d\mu(x,v)\Bigr\}.
\end{align*}
By \eqref{alpha2}, the map $(x',v')\mapsto\int_{\Gamma_+}
\phi(x,v)(\alpha_2-\tilde{\alpha}_2)(x,v,x',v')\:d\mu(x,v)$ is in
$L^\infty(\Gamma_-)$, and then, by \eqref{app_id_limit}, we get for almost
every $(x_0',v_0')\in\Gamma_-$
\begin{align}
  I_2(\phi)(x_0',v_0')
  :=\lim_{\eps\to 0}I_2(\phi,\eps)
  =\int_{\Gamma_+}\phi(x,v)(\alpha_2-\tilde{\alpha}_2)(x,v,x_0',v_0')
  \:d\mu(x,v).\label{term3_lim}
\end{align}
The left hand side of \eqref{original_estimate} has three terms. We
move the third term to the right hand side (with absolute values)
and take the limit with $\eps\to 0$ to get
\begin{align}
  \label{original_estimate2}
  |I_0(\phi)+I_1(\phi)|(x_0',v_0')
  \leq \|\A-\tA\|+I_2(|\phi|)(x_0',v_0'),~a.e.~(x_0',v_0')\in\Gamma_-,
\end{align}
for any $\phi\in L^\infty(\Gamma_+)$ with $\|\phi\|_\infty=1$.

We note that the negligible set on which the inequality above does not hold
may depend on $\phi$.  We will consider a countable sequence of functions
$\phi$, and since the countable union of negligible sets is negligible, the
inequality \eqref{original_estimate2} holds almost everywhere on
$\Gamma_-$, independently of the term in the sequence. This justifies the
argument below for almost every $(x_0',v_0')$ in $\Gamma_-$.

In \eqref{original_estimate2}, we shall choose two sequences of $\phi$ to
conclude the two estimates of the lemma.  First we show the estimate
\eqref{bal_estimate1} by choosing $\phi_m\in L^\infty(\Gamma_+)$ which are
1 in a shrinking neighborhood of
$(x_0,v_0):=\geov{x_0',v_0'}{\tau_+(x_0',v_0')}$.  First, define
$\phi_m(x_0,v)$ to be the indicator function for the set $\{v\in
S_{x_0}M\st \|v-v_0\|_{g(x_0)}<1/m\}$, then extend $\phi_m$ by
\begin{align*}
  \phi_m(x,v)=
  \begin{cases}
    0&\text{if }d_{\p M}(x,x_0)\geq1/m,\\
    \phi_m(x,\P(v;x,x_0))&\text{if }d_{\p M}(x,x_0)<1/m
  \end{cases}
\end{align*}
where $\P(v;x,x_0)$ is the parallel transport of $v\in S_xM$ from $x$ to
$x_0$. Then \eqref{term1_lim} gives
\begin{align*}
  I_0(\phi_m)=e^{-\int_0^{\tau_+(x_0',v_0')}a(\geov{x_0',v_0'}{s})\:ds}-
  e^{-\int_0^{\tau_+(x_0',v_0')}\tilde{a}(\geov{x_0',v_0'}{s})\:ds}
\end{align*}
independently of $m$. From \eqref{term2_lim} we have $\lim_{m\to
  \infty}I_1(\phi_m)=0$ since for any $t$, the support of
$\phi\bigl(\geov{y(t),w}{\tau_+(y,w)}\bigr)$ in $w\in S_{y(t)}M$
shrinks to $\dot y(t)$. From \eqref{term3_lim} we also have
$\lim_{m\to\infty}I_2(|\phi_m|)=0$, since the support shrinks to one
point. We use here the corollary of \eqref{alpha2} that
$(\alpha_2-\tilde \alpha_2)(\cdot,\cdot,x_0',v_0')\in
L^1(\Gamma_+)$, for a.e. $(x_0',v_0')\in\Gamma_-$.

Next we prove the estimate \eqref{bal_estimate2}.  Recall that now
$n\geq3$.  For $m>0$, let $N_{(x_0',v_0'),q}\subset \overline{M}$ be the
tubular neighborhood of the geodesic $y(t)=\geo{x_0',v_0'}{t}$, $0\leq
t\leq \tau_+(x_0',v_0')$, of radius $1/m$.  We now define a sequence
$\phi_m\in L^\infty(\Gamma_+)$.  Set $\phi_m(x,v)=0$ if $x\in
N_{(x_0',v_0'),q}$; note that $I_0(\phi_m)=0$ for all $m$.  For
$(x,v)\in\Gamma_+$ with $x\not\in N_{(x_0',v_0'),q}$, $\phi_q(x,v)=0$ if
the geodesic $z(s)=\geo{x,v}{s}$, $-\tau_-(x,v)\leq s\leq 0$ does not
intersect $N_{(x_0',v_0'),q}$.  When $\geo{x,v}{\cdot}$ does intersect
$N_{(x_0',v_0'),q}$, let $0\leq t(x,v)\leq \tau_+(x_0',v_0')$ and
$-\tau_-(x,v)\leq s(x,v)\leq 0$ be such that
\begin{align*}
  d_g\bigl(y(t(x,v)),z(s(x,v))\bigr)
  =\min_{s,t}\{d_g(y(t),z(s)\}
\end{align*}
and define
\begin{align*}
  \phi_m(x,v)=\text{sgn}(k-\tk)\bigl(y(t(x,v)),\dot y(t(x,v)),
  \P\bigl(\dot z(s(x,v));z(s(x,v)),y(t(x,v))\bigr)\bigr).
\end{align*}
Notice that when $(x,v)$ is of the form $\geov{y(t),w}{\tau_+(y(t),w)}$,
$w\in S_{y(t)}M$, that is $(x_0',v_0')$ and $(x,v)$ are the beginning and
end of a single-scattering broken geodesic, $\phi_m(x,v)$ takes the sign of
$k-\tilde k$ at the point of scattering.  Note also that the support of
$\phi_m$ shrinks to a negligible set in $\Gamma_+$ as $m\to\infty$ since
$n\geq3$.

Now apply the estimate \eqref{original_estimate2} to $\phi_m$ and use
$I_0(\phi_m)=0$ to get
\begin{align*}
  |I_{1,1}(\phi_m)|(x_0',v_0')
  \leq\|\A-\tA\|+I_2(|\phi_m|)(x_0',v_0')+|I_{1,2}(\phi_m)|(x_0',v_0').
\end{align*}
Since the support of $\phi_m$ shrinks to a set of measure zero in
$\Gamma_+$ as $m\to\infty$, we get for almost every
$(x_0',v_0')\in\Gamma_-$, $\lim_{m\to\infty}I_3(|\phi_m|)(x_0',v_0')=0$.
Finally, noting that $|I_{1,1}(\phi_m)|=I_{1,1}(\phi_m)$ and applying
\eqref{term22_lim}, from \eqref{term21_lim} we obtain for almost every
$(x_0',v_0')\in\Gamma_-$
\begin{align*}
  &\int_0^{\tau_+(x_0',v_0')}\int_{S_{y(t)}M}
  |k-\tilde{k}|(y(t),\dot y(t),w)
  E(x_0',y)E\bigl(y(t),\geo{y(t),w}{\tau_+(y(t),w)}\bigr)\:d\omega_y(w)\:dt\\
  &\ =\lim_{m\to \infty}I_{1,1}(\phi_m)
  \leq \|\A-\tA\|+
  \int_0^{\tau_+(x_0',v_0')}\int_{S_yM}
  |F-\tF|(x_0',y,w)\tk(y,\dot y,w)\:d\omega_y(w)\:dt.
\end{align*}
The estimate \eqref{bal_estimate2} in the theorem follows.
\end{proof}

\section{Stability modulo gauge transformations}
In this section we prove Theorem \ref{main_thm}.

We start with two pairs $(a,k),(\ta,\tk)\in U_{\Sigma,\rho}$ and let
$$\eps:=\norm\A-\tA\norm.$$ We shall find an intermediate pair
$(a',k')\sim (a,k)$ such that \eqref{closeness_in_a} and
\eqref{closeness_in_k} hold.

Define first the ``trial'' gauge transformation:
\begin{align}
  \label{trial_gauge}
  \varphi(x,v)
  :=e^{-\int_0^{\tau_-(x,v)}(\ta-a)(\geov{x,v}{s-\tau_-(x,v)})\:ds},
  \quad a.e.\ (x,v)\in \overline {SM}.
\end{align}
Then $\varphi>0$, $\varphi|_{\Gamma_-}=1$, $\D\varphi(x,v)\in
L^\infty(SM)$ and
\begin{align}\label{a1-a}
  \ta(x,v)=a(x,v)-\D\log\varphi(x,v).
\end{align}
Note, however, that $\varphi|_{\Gamma_+}$ is not equal to 1. We begin by
estimating $\varphi|_{\Gamma_+}$.  By \eqref{bal_estimate1}, we have for
almost every $(x_0',v_0')\in\Gamma_-$
\begin{align*}
  \Bigl|e^{-\int_0^{\tau_+(x_0',v_0')}a(\geov{x_0',v_0'}{s})\:ds}-
  e^{-\int_0^{\tau_+(x_0',\theta_0')}\tilde a(\geov{x_0',v_0'}{s})\:ds}\Bigr|
  \leq\eps.
\end{align*}
Changing variables $t=\tau_+(x_0',v_0')-s$ and denoting
$(x_0,v_0)=\geov{x_0',v_0'}{\tau_+(x_0',v_0')}$ we get
\begin{align}
  \label{trial_gauge_estimate1}
  \Bigl|
  e^{-\int_0^{\tau_-(x_0,v_0)}a(\geov{x_0,v_0}{t-\tau_-(x_0,v_0)})\:dt}-
  e^{-\int_0^{\tau_-(x_0,v_0)}\ta(\geov{x_0,v_0}{t-\tau_-(x_0,v_0)})\:dt}
  \Bigr|\leq\eps.
\end{align}
When $(x_0',v_0')$ covers $\Gamma_-$ almost everywhere we get
$(x_0,v_0)$ covers $\Gamma_+$ almost everywhere.

By the Mean Value theorem applied to $u\mapsto e^{-u}$ we obtain the lower
bound
\begin{align}
  &\Bigl|
  e^{-\int_0^{\tau_-(x_0,v_0)}a(\geov{x_0,v_0}{t-\tau_-(x_0,v_0)})\:dt}-
  e^{-\int_0^{\tau_-(x_0,v_0)}\ta(\geov{x_0,v_0}{t-\tau_-(x_0,v_0)})\:dt}
  \Bigr|\nonumber\\
  &\qquad=e^{-u_0}\Bigl|\int_0^{\tau_-(x_0,v_0')}
  (\ta-a)\bigl(\geov{x_0,v_0}{t-\tau_-(x_0,v_0)}\bigr)\:dt\Bigr|
  =e^{-u_0}|\log\varphi(x_0,v_0)|\nonumber\\
  &\qquad
  \geq e^{-\diam(M)\Sigma}|\log\varphi(x_0,v_0)|
  \label{trial_gauge_estimate2}
\end{align}
where $u_0=u_0(x_0,v_0,a,\tilde{a})$ is a value between the two integrals
appearing in the exponents in the left hand side above, and $\varphi$ is
defined in \eqref{trial_gauge}.

From \eqref{trial_gauge_estimate1} and \eqref{trial_gauge_estimate2} we get
the following estimate for the ``trial'' gauge $\varphi$:
\begin{align}
  \label{trial_gauge_less_epsilon}
  |\log\varphi(x,v)|\leq
  e^{\diam(M)\Sigma}\eps,~a.e.~(x,v)\in\Gamma_+.
\end{align}

The ``trial'' gauge $\varphi$ is not good enough since it does not
equal 1 on $\Gamma_+$. We alter it to some $\tilde{\varphi}\in
L^\infty(SM)$ with $\D\log\tilde{\varphi}\in L^\infty(SM)$ in such a
way that $\tilde{\varphi}|_{\p SM}=1$. More precisely, for almost
every $(x,\theta)\in\overline{SM}$, we define $\tilde{\varphi}(x,v)$
by
\begin{align}\label{def_gauge}
  \log\tilde{\varphi}(x,v):=\log\varphi(x,v)-
  \frac{\tau_-(x,v)}{\tau(x,v)}\log\varphi\bigl(\geov{x,v}{\tau_+(x,v)}\bigr).
\end{align}
Since $0\leq\tau_-(x,v)/\tau(x,v)\leq 1$ we get $\tilde{\varphi}\in
L^\infty(SM)$, and clearly $\tilde{\varphi}|_{\p SM}=1$.  Since
$\D\tau(x,v)=\D\log\varphi\bigl(\geov{x,v}{\tau_+(x,v)}\bigr)=0$ and
$\D\tau_-(x,v)=1$,
\begin{align}
  \label{closeness_varphi}
  \D\log\tilde{\varphi}(x,v)=\D\log\varphi(x,v)-
  \frac{\log\varphi\bigl(\geov{x,v}{\tau_+(x,v)}\bigr)}{\tau(x,v)}
  \in L^\infty(SM).
\end{align}
Define now the pair $(a',k')$ in the equivalence class of $\langle
a,k\rangle$ by
\begin{align}
  a'(x,v)&:=a(x,v)-\D\log\tilde{\varphi}(x,v),\label{tildea-a}\\
  k'(x,v',v)&:=\frac{\tilde{\varphi}(x,v)}{\tilde{\varphi}(x',v')}
  k(x,v',v).\label{tildek-k}
\end{align}
Now $\A'$, the albedo operator corresponding to $(a',k')$,
satisfies $\A'=\A$, and
\begin{align*}
\norm\A'-\tA\norm=\norm\A-\tA\norm=\eps.
\end{align*}

Next we compare the pairs $(a',k')$ with $(\ta,\tk)$ and show them
to satisfy \eqref{closeness_in_a} and \eqref{closeness_in_k}.  Using
the definitions \eqref{a1-a}, \eqref{tildea-a}, the relation
\eqref{closeness_varphi}, and the estimate
\eqref{trial_gauge_less_epsilon} for $\varphi$ on $\Gamma_+$, we
have for almost every $(x,v)\in SM$:
\begin{align}
  |\ta(x,v)-a'(x,v)|
  &=\bigl|[\ta-a](x,v)+[a-a'](x,v)\bigr|\nonumber\\
  &=|\D\log\tilde{\varphi}(x,v)-\D\log{\varphi}(x,v)|\nonumber\\
  &=\frac{\bigl|\log\varphi
    \bigl(\geov{x,v}{\tau_+(x,v)}\bigr)\bigr|}{\tau(x,v)}
  \leq\eps\frac{e^{\diam(M)\Sigma}}{\tau(x,v)}.\label{a1-tildea}
\end{align}
Since the coefficients are supported away from $\partial M$ (by
construction of $M$) such that \eqref{nontangential} holds,
following \eqref{a1-tildea} we obtain the estimate
\eqref{closeness_in_a} in the form
\begin{align}
  \label{final_estimate_for_a}
  \|\ta-a'\|_\infty\leq \eps\frac{ e^{\diam(M)\Sigma}}{c_0},
\end{align}
with $c_0$ from \eqref{nontangential}.

Up to this point, all the arguments above also work for two
dimensional domains. Next we prove the estimate
\eqref{closeness_in_k}. These arguments are specific to three or
higher dimensions. Recall the formula \eqref{F defn} adapted to
$a'$: let $x'\in\p M$, $y\in M$ and $w\in S_yM$ and let $v'\in
S_{x'}M,t>0$ be such that $y=\geo{x',v'}{t}$.  Then from
\eqref{final_estimate_for_a},
\begin{align*}
  |a'(x,v)|\leq \eps\frac{ e^{\diam(M)\Sigma}}{\tau(x,v)}+\Sigma,
  \qquad\text{and}\qquad
  \|a'\|_\infty\leq \eps\frac{ e^{\diam(M)\Sigma}}{c_0}+\Sigma,
\end{align*}
so (using the first, and the fact that $\tau$ is constant along geodesics),
\begin{align}
  |F'(x',y,w)|
  &=e^{-\int_0^ta'(\geov{x',v'}{s})\:ds}
  e^{-\int_0^{\tau_+(y,w)}a'(\geov{y,w}{s})\:ds}\nonumber\\
  &\geq
  \exp\bigl(-2(\varepsilon e^{\diam(M)\Sigma}+\diam(M)\Sigma\bigr).
  \label{lower_bound_E}
\end{align}
Using the non-negativity of $\ta$ and $a'$ we estimate
\begin{align}
  |[\tF-F']&(x',y,w)|
  \leq \Bigl|e^{-\int_0^t\ta(\geov{x',v'}{s})\:ds}
  -e^{-\int_0^t\ta(\geov{x',v'}{s})\:ds}\Bigr|\nonumber\\
  &\qquad\qquad\qquad
  +\Bigl|e^{-\int_0^{\tau_+(y,w)}a'(\geov{y,w}{s})\:ds}
  -e^{-\int_0^{\tau_+(y,w)}a'(\geov{y,w}{s})\:ds}\Bigr|\nonumber\\
  &\leq\Bigl|\int_0^t [\ta-a'](\geov{x',v'}{s})\:ds\Bigr|+
  \Bigl|\int_0^{\tau_+(y,w)}[\ta-a'](\geov{y,w}{s})\:ds\Bigr|\nonumber\\
  &\leq \eps e^{\diam(M)\Sigma}\Bigl(\frac{t}{\tau(x',v')}
  +\frac{\tau_+(y,w)}{\tau(y,w)}\Bigr)
  \leq 2\eps e^{\diam(M)\Sigma}
  \label{F-F}
\end{align}
by \eqref{a1-tildea}.  We now apply the lower bound for $F'$ in
\eqref{lower_bound_E}, the upper bound for $\|\tilde F-F'\|_\infty$ from
\eqref{F-F} and the hypothesis $\|\tk\|_{\infty,1}\leq \rho$ to the estimate
\eqref{bal_estimate2} with respect to the pairs $(a',k')$ and $(\ta,\tk)$.
With $y(t)=\geo{x_0',v_0'}{t}$, we obtain
\begin{align*}
  &\int_0^{\tau_+(x_0',v_0')}\int_{S_{y(t)}M}
  |k-\tilde{k}|(y(t),\dot y(t),w)\:d\omega_y(w)\:dt\\
  &\qquad\leq \varepsilon
  \bigl(1+2\diam(M)\rho\omega_{n-1}e^{\diam(M)\Sigma}\bigr)
  \exp\Bigl(2\diam(M)
  \bigl(\eps\frac{ e^{\diam(M)\Sigma}}{c_0}+\Sigma\bigr)\Bigr)\\
  &\qquad=\varepsilon\:C_1,\quad\text{say.}
\end{align*}
Finally, integrating the formula above in $(x_0',v_0')\in\Gamma_-$ with the
measure $d\mu(x_0',v_0')$, we get
\begin{align*}
  \|\tk-k'\|_1\leq \eps \text{Vol}(\p M)\omega_{n-1}C_1.
\end{align*}
Theorem \ref{main_thm} holds now with $C=\max\{\text{Vol}(\p
M)\omega_{n-1}C_1,e^{\diam(M)\Sigma}/c_0\}$.

\section{Stability of the equivalence classes in two dimensions}
We prove here Theorem \ref{main_thm_2d}, making use of the results of
\cite{Mc05}.

Let $(a,k),(\ta,\tk)\in V_{\Sigma,\rho}$ be given with $\|\A-\tA\|_*=\eps$.
As before, define the pair $(a',k')$ in the equivalence class of $\langle
a,k\rangle$ by \eqref{tildea-a} and \eqref{tildek-k}.  Then the
corresponding albedo operator $\A'=\A$ and, thus,
\begin{align}
  \label{beta-beta}
  \|\A'-\tA\|_*=\|\A-\tA\|_*=\eps, \quad\text{ and }\quad
  \|(\tilde{\beta}-\beta')|\sin\psi|\|_\infty\leq\eps.
\end{align}
The estimate \eqref{final_estimate_for_a} holds as in the case $n\geq3$.
Now
\begin{align}
  \label{quotient_varphi-pre}
  \frac{\tilde{\varphi}(x,v)}{\tilde{\varphi}(x',v')}
  =e^{-\int_0^{\tau_-(x,v)}(a'-a)(\geov{x,v}{s-\tau_-(x,v)})\:ds
    +\int_0^{\tau_-(x',v')}(a'-a)(\geov{x',v'}{s-\tau_-(x',v')})\:ds}.
\end{align}
From \eqref{trial_gauge_less_epsilon} and \eqref{a1-tildea}, with
$y(s)=\geo{x,v}{s-\tau_-(x,v)}$,
\begin{align*}
  \Bigl|\int_0^{\tau_-(x,v)}(a'-a)(y,\dot y)\:ds\Bigr|
  &\leq
  \int_0^{\tau_-(x,v)}(|a'-\tilde a|+|\tilde a-a|)(y,\dot y)\:ds\\
  &\leq
  \int_0^{\tau_-(x,v)}
  \varepsilon\frac{e^{\diam(M)\Sigma}}{\tau(y,\dot y)}\:ds
  +\varepsilon e^{\diam(M)\Sigma}
  \leq 2\varepsilon e^{\diam(M)\Sigma};
\end{align*}
the same estimate holds for the second exponent in
\eqref{quotient_varphi-pre}.  Thus, from the definition \eqref{tildek-k}
and \eqref{quotient_varphi-pre} we obtain
\begin{align}
  \label{rho'}
  \frac{\tilde{\varphi}(x,v)}{\tilde{\varphi}(x',v')}\leq
  \exp\bigl\{4\varepsilon e^{\diam(M)\Sigma}\bigr\}
  \implies
  \|k'\|_\infty
  \leq \rho\exp\bigl\{4\varepsilon e^{\diam(M)\Sigma}\bigr\}.
\end{align}

Let
\begin{align*}
 \tilde E_1(y,w',w)
 :=\tilde E\bigl(\geo{y,w'}{-\tau_-(y,w')},y\bigr)
 \tilde E\bigl(y,\geo{y,w}{\tau_+(y,w)}\bigr)
\end{align*}
be the total attenuation along the broken geodesic due to one scattering at
$(y,w',w)\in S^2M$.  Then \eqref{lower_bound_E} and \eqref{F-F} say
\begin{align}
  |E'_1(y,w',w)|
  &\geq
  \exp\bigl(-2(\varepsilon e^{\diam(M)\Sigma}+\diam(M)\Sigma)\bigr)
  =C_1,\text{ say, and}
  \label{estimate_on_F}\\
  \|\tilde E_1-E_1'\|&\leq 2\eps e^{\diam(M)\Sigma}.
  \label{estimate_on_F-F}
\end{align}

The terms $\alpha_j$ in the expansion of the albedo kernel in Proposition
\ref{albedo kernel 2} are the traces of distributions $\phi_j$ defined on
$SM\times\Gamma_-$; see \cite{Mc05}.  The $\phi_j$ are the kernels of the
operators $J$, $KJ$ and $(I-K)^{-1}K^2J$ ($j=0,1,2$, respectively) where
\begin{align*}
  Jf_-(x,v)&=E(\geo{x,v}{-\tau_-(x,v)},x)f_-(\geov{x,v}{\tau_-(x,v)}),\\
  Kf(x,v)&=\int_0^{\tau_-(x,v)}E\bigl(x,\geo{x,v}{t-\tau_-(x,v)}\bigr)
  T_1f\bigl(\geov{x,v}{t-\tau_-(x,v)}\bigr)\:dt,\text{ with}\\
  T_1f(x,v)&=\int_{S_xM}k(x,v',v)f(x,v')\:d\omega_x(v').
\end{align*}
Let $\gamma$ be the trace operator on $L^\infty(\Gamma_+)$, which is shown
in \cite{Mc05} to be well-defined.

Let $(x,v,x',v')\in\Gamma_+\times\Gamma_-$ be such that the
geodesics $\geo{x,v}{\cdot}$ and $\geo{x',v'}{\cdot}$ intersect at
$(y,w',w)\in S^2M$.  By Proposition \ref{albedo kernel 2} above,
\begin{align*}
  E_1'(\tk-k')
  &=(E_1'-\tilde E_1)\tilde k+(\tilde E_1\tilde k-E_1'k')\\
  &=(E_1'-\tilde E_1)\tk+(\tilde \beta-\beta')|\sin\psi|+
  (\gamma\phi_2'-\gamma\tilde \phi_2)|\sin\psi|,
\end{align*}
and so by \eqref{beta-beta}, \eqref{rho'}, \eqref{estimate_on_F} and
\eqref{estimate_on_F-F},
\begin{align}
  \label{foundation}
  C_1|\tk-k'|&\leq 2\varepsilon e^{\diam(M)\Sigma}\rho+\varepsilon
  +|\gamma\phi_2'-\gamma\tilde \phi_2||\sin\psi|.
\end{align}
Now
\begin{align}
  |\gamma\phi_2'-\gamma\phi_2|
  &=\gamma(I-K')^{-1}K'^2\phi_0'
  -\gamma(I-\tilde K)^{-1}\tilde K^2\tilde \phi_0\nonumber\\
  &=\gamma(I-K')^{-1}\bigl(K'^2\phi_0'-\tilde K^2\tilde \phi_0\bigr)
  +\gamma(I-\tilde K)^{-1}(K'-\tilde K)(I-K')^{-1}\tilde K^2\tilde \phi_0
  \nonumber\\
  &=\gamma(I-K')^{-1}\bigl(K'^2-\tilde K^2\bigr) \phi'_0
  +\gamma(I-K')^{-1}\tilde K^2(\phi'_0-\tilde \phi_0)\nonumber\\
  &\qquad
  +\gamma(I-\tilde K)^{-1}(K'-\tilde K)(I-K')^{-1}\tilde K^2\tilde \phi_0.
  \label{estimate_alpha2}
\end{align}
Lemma 9 of \cite{Mc05} estimates the first of these terms:
\begin{align}
  \label{piece 1}
  \|\gamma(I-K')^{-1}
  \bigl(K'^2-\tilde K^2\bigr)\phi'_0\|_\infty
  \leq C_2\|k'-\tilde k\|_\infty(\|k'\|_\infty+\|\tilde k\|_\infty)
  (1-\log|\sin\psi|).
\end{align}
For the second we appeal to Proposition 7, and its proof, in \cite{Mc05}.
Instead of using the estimate $\|E\|_\infty\leq 1$, we use \eqref{F-F};
together with the fact that $(I-K')^{-1}$ preserves $L^\infty$, we readily
obtain
\begin{align}
  \label{piece 2}
  \|\gamma(I-K')^{-1}\tilde K^2(\phi'_0-\tilde \phi_0)
  \leq C_3\|\tilde k\|_\infty^2\|E'-\tilde E\|_\infty
  \leq C_3\rho^22\varepsilon e^{\diam(M)\Sigma}.
\end{align}
The final term in \eqref{estimate_alpha2} is estimated in \cite[Lemma
10]{Mc05}:
\begin{align}
  \label{piece 3}
  \|\gamma(I-\tilde K)^{-1}(K'-\tilde K)(I-K')^{-1}
  \tilde K^2\tilde \phi_0\|_\infty
  \leq C_4\|k'-\tk\|_\infty\rho^2(1+\rho).
\end{align}
Combining \eqref{foundation}, \eqref{piece 1}, \eqref{piece 2} and
\eqref{piece 3}, for a new constant $C$, we obtain
\begin{align*}
  \|k'-\tilde k\|_\infty\leq C\varepsilon+C\rho\|k'-\tilde k\|_\infty
\end{align*}
and so if $\rho<1/C$, we obtain the final estimate
\begin{align*}
  \|k'-\tilde k\|_\infty\leq
  \frac{C}{1-C\rho}\varepsilon=:\tilde{C}\varepsilon.
\end{align*}

Theorem \ref{main_thm_2d} is now proven for a constant which is the
maximum of the $\tilde{C}$ above and $\exp(\diam(M)\Sigma)/{c_0}$
(see \eqref{final_estimate_for_a}).

\section*{Acknowledgment}This work originated in discussions during
the BIRS-workshop {\em Inverse Transport Theory and Tomography},
Banff, Alberta, Canada, May 16-21, 2010.


\begin{thebibliography}{99}






\bibitem{Ba06} 
  \newblock {G. Bal},
  \newblock {\em Radiative transfer equations with varying
    refractive index: a mathematical perspective},
  \newblock IJ. Opt. Soc. Amer. A {\bf 23} (2006), 1639--1644.

\bibitem{Ba09} 
  \newblock {G. Bal},
  \newblock {\em Inverse Transport theory and applications},
  \newblock Inverse Problems {\bf 25} (2009), 053001, 48 pp.

\bibitem{BaJo08}
  \newblock {G.~Bal and A. Jollivet},
  \newblock {\em Stability estimates in stationary inverse transport}
  \newblock Inverse Problems {\bf 25} (2009), 075010, 32 pp.

\bibitem{BaLaMo08} 
  \newblock {G.~Bal, I. Langmore and F. Monard},
  \newblock {\em Inverse transport with isotropic sources and angularly
    averaged measurements},
  \newblock Inverse Probl. Imaging {\bf 2} (2008), 23--42.

\bibitem{Bo}
  \newblock {A. Bondarenko},
  \newblock {\em The structure of the fundamental solution of the
    time-independent transport equation},
  \newblock {J. Math. Anal. Appl. {\bf 221}(1998), no. 2, 430--451.}


\bibitem{ChSt96} 
  \newblock {M. Choulli and P. Stefanov},
  \newblock {\em Inverse scattering and inverse boundary value problems for
    the linear Boltzmann equation},
  \newblock Comm. P.D.E. {\bf 21} (1996), 763--785.

\bibitem{ChSt99} 
  \newblock {M. Choulli and P. Stefanov},
  \newblock {\em An inverse boundary value problem for the stationary
    transport},
  \newblock Osaka J. Math. {\bf 36} (1999), 87--104.

\bibitem{DaLi93} 
  \newblock {R.~Dautray and J.-L. Lions},
  \newblock Mathematical Analysis and Numerical Methods for Science and
  Technology. Vol.6,
  \newblock Springer Verlag, Berlin, 1993.

\bibitem{folland}
  \newblock {G. Folland},
  \newblock {``Real Analysis, Modern Techniques and Their Applications''},
  \newblock {John Wiley \& Sons, New York, 1984.}


\bibitem{La08} 
    \newblock {I. Langmore},
    \newblock {\em The stationary transport problem with angularly averaged
      measurements},
    \newblock Inverse Problems {\bf 24} (2008), 015024, 22 pp.

\bibitem{LaMc08} 
  \newblock {I. Langmore and S. McDowall},
  \newblock {\em Optical tomography for variable refractive index with
    angularly averaged measurements},
  \newblock Comm. PDE. {\bf 33} (2008), 2180--2207.



\bibitem{Wang98}
    \newblock{G. Marquez, L. V. Wang, S.-P. Lin, J. A. Schwartz, and
    S. L. Thomsen}, 
    \newblock{\em Anisotropy in the absorption and scattering spectra of
    chicken breast tissue}, 
    \newblock {Applied Optics {\bf 37} (1998), 798-805.}


\bibitem{Mc04} 
  \newblock {S. McDowall},
  \newblock {\em An inverse problem for the transport equation in the
    presence of a Riemannian metric},
  \newblock Pac.~J.~Math. {\bf 216} (2004), 107--129.

\bibitem{Mc05} 
  \newblock {S. McDowall},
  \newblock {\em Optical tomography on simple Riemannian surfaces},
  \newblock Comm.~PDE. {\bf 30} (2005), 1379--1400.

\bibitem{Mc08}
  \newblock {S. McDowall},
  \newblock {\em Optical tomography for media with variable index of
    refraction},
  \newblock CUBO {\bf 11} (2009), 71--98.

\bibitem{McStTa10-GaEq}
  \newblock {S. McDowall, P. Stefanov and A. Tamasan},
  \newblock {\em Gauge equivalence in stationary radiative transport
  through media with varying index of refraction},
  \newblock Inverse Problems and Imaging {\bf 4} (2010), no. 1, 151--168.

\bibitem{McStTa10-GaSt}
  \newblock {S. McDowall, P. Stefanov and A. Tamasan},
  \newblock {\em Stability of the gauge equivalent classes in stationary
  inverse transport},
  \newblock Inverse Problems {\bf 26} (2010), 025006, 19pp.

\bibitem{MoKh97}
  \newblock {M.~Mokhtar-Kharroubi},
  \newblock Mathematical Topics in Neutron Transport Theory,
  \newblock World Scientific, Singapore, 1997.

\bibitem{ReSi79} 
  \newblock {M. Reed and B. Simon},
  \newblock Methods of Modern Mathematical Physics, Vol. 3,
  \newblock Academic Press, New York, 1979.

\bibitem{Kui10}
\newblock {Kui Ren},
\newblock {\em Recent developments in numerical techniques for
transport-based medical imaging methods},
  \newblock {Comm. Comp. Phys. \textbf{8}(1)(2010), 1--50.}

\bibitem{Ro66} 
  \newblock {V. Romanov},
  \newblock{\em Stability estimates in problems of recovering the
    attenuation coefficient and the scattering indicatrix for the transport
    equation},
  \newblock J. Inverse Ill-Posed Probl. \textbf{4} (1996), 297--305.

\bibitem{Sh94} 
  \newblock Sharafutdinov, V. A.,
  \newblock Integral geometry of tensor fields,
  \newblock Inverse and ill-posed problems series, VSP, The Netherlands, 1994.

\bibitem{StTa09} 
  \newblock {P. Stefanov and A. Tamasan},
  \newblock {\em Uniqueness and non-uniqueness in inverse radiative transfer},
  \newblock Proc. Amer. Math. Soc. {\bf 137} (2009), 2335--2344.

\bibitem{StUh03} 
  \newblock {P.~Stefanov and G.~Uhlmann},
  \newblock {\em Optical tomography in two dimensions},
  \newblock Methods Appl. Anal. {\bf 10} (2003), 1--9.

\bibitem{StUh04} 
  \newblock {P.~Stefanov and G.~Uhlmann},
  \newblock {\em Stability estimates for the X-ray transform of tensor
  fields and boundary rigidity}, 
  \newblock Duke Math. J. {\bf 123} (2004), no. 3, 445--467.


\bibitem{Ta02} 
  \newblock {A.~Tamasan},
  \newblock {\em An inverse boundary value problem in two-dimensional
    transport},
  \newblock Inverse Problems {\bf 18} (2002), 209--219.

\bibitem{Ta03} 
  \newblock {A.~Tamasan},
  \newblock {\em Optical tomography in weakly anisotropic scattering media},
  \newblock Contemporary Mathematics {\bf 333} (2003), 199--207.

\bibitem{Wa99} 
  \newblock J.-N. Wang,
  \newblock {\em Stability estimates of an inverse problem for the
    stationary transport},
  \newblock Ann. Inst. Henri Poincar\'e {\bf 70} (1999), 473--495.
\end{thebibliography}
\end{document}